\documentclass[a4paper,reqno,11pt]{article}

\usepackage[utf8]{inputenc}
\usepackage[T1]{fontenc}
\usepackage{xspace,
amsthm,
amsmath,
amssymb,
bbm,
tikz,
graphicx,
subfigure,
keyval}
\DeclareGraphicsExtensions{.png, .pdf}

\usepackage[hidelinks]{hyperref}

\usepackage{geometry}
\tikzstyle{nodino}=[circle,draw,fill,inner sep=0pt,minimum size=0.5mm]
\tikzstyle{infinito}=[circle,inner sep=0pt,minimum size=0mm]
\tikzstyle{nodo}=[circle,draw,fill,inner sep=0pt, minimum size=0.5*width("k")]
\tikzstyle{nodo_vuoto}=[circle,draw,inner sep=0pt, minimum size=0.5*width("k")]
\usetikzlibrary{graphs}
\tikzset{every loop/.style={min distance=10mm,in=300,out=240,looseness=10}}
\tikzset{place/.style={circle,thick,draw=blue!75,fill=blue!20,minimum
size=6mm}}
\tikzset{place2/.style={circle,thick,draw=red!75,fill=red!20,minimum
size=6mm}}

\newcommand{\rr}{{\mathbb R}}

\newcommand{\G}{{\mathcal{G}}}

\newcommand{\udotT}{\|u'\|_{L^2(\mathcal{T})}}

\newcommand{\uLtwoT}{\|u\|_{L^2(\mathcal{T})}}
\newcommand{\HmuT}{H_\mu^1(\mathcal{T})}
\newcommand{\uLsixT}{\|u\|_{L^6(\mathcal{T})}}

\newcommand{\T}{\mathcal{T}}
\newcommand{\K}{\mathcal{K}}
\newcommand{\h}{\mathcal{H}}

\newcommand{\dx}{\,dx}

\theoremstyle{plain} 
\newtheorem{thm}{Theorem}[section] 
 
\newtheorem{lem}[thm]{Lemma}

\theoremstyle{definition}

\theoremstyle{definition}

\theoremstyle{remark} 
\newtheorem{rem}{Remark}[section]

\title{Ground states of the $L^2$-critical NLS equation with localized nonlinearity on a tadpole graph}

\author{Simone Dovetta$^{\dagger,\sharp}$ and Lorenzo Tentarelli$^\ddagger$
\\ \ \\ \ \\
{\small  $^\dagger$Dipartimento di Scienze Matematiche ``G.L. Lagrange'' } \\
{\small Politecnico di Torino } \\
{\small Corso Duca degli Abruzzi, 24, 10129 Torino, Italy} \\ \ \\
{\small $^\sharp$Dipartimento di Matematica ``G. Peano''}\\
{\small Universit\`a degli Studi di Torino }\\
{\small Via Carlo Alberto, 10, 10123, Torino, Italy} \\
{\small \texttt{simone.dovetta@polito.it}}\\ \ \\
{\small $^\ddagger$Dipartimento di Matematica} \\
{\small Sapienza Universit\`a di Roma } \\
{\small Piazzale Aldo Moro, 5, 00185 Roma, Italy} \\
{\small \texttt{tentarelli@mat.uniroma1.it}}
}

\begin{document}

\maketitle

\begin{abstract}
 The paper aims at giving a first insight on the existence/nonexistence of ground states for the $L^2$-critical NLS equation on metric graphs with localized nonlinearity. As a consequence, we focus on the tadpole graph, which, albeit being a toy model, allows to point out some specific features of the problem, whose understanding will be useful for future investigations. More precisely, we prove that there exists an interval of masses for which ground states do exist, and that for large masses the functional is unbounded from below, whereas for small masses ground states cannot exist although the functional is bounded.
\end{abstract}

\smallskip



\section{Introduction}

The study of evolution equations on \emph{metric graphs} or networks has gained a great popularity in recent years, since they represent effective models for the study of the dynamics of physical systems living in branched spatial structures (see, e.g., \cite{AST5} and the references therein). More precisely, a particular interest has been addressed to the investigation of the \emph{focusing} nonlinear Schr\"odinger (a.k.a. NLS) equation, namely
\begin{equation}
 \label{eq-NLStime}
 \imath\dot{\psi}=-\psi''-|\psi|^{p-2}\,\psi\qquad(p\geq2)
\end{equation}
with suitable boundary conditions at the vertices of the graph, as it is supposed to well approximate, for $p=4$, the behavior of Bose-Einstein condensates in ramified traps (see, e.g., \cite{GW}).

\medskip
From the mathematical point of view, the discussion is mainly focused on the study of the stationary solutions of \eqref{eq-NLStime}, that is functions of the form $\psi(t,x)=e^{i\lambda t}\,u(x)$, with $\lambda\in\rr$, solving the stationary equation associated to \eqref{eq-NLStime}

\begin{equation*}
   \label{eq-statNLSgeneral}
   u''+|u|^{p-2}\,u=\lambda u\,.
\end{equation*}

In this perspective, the first pioneering works (e.g., \cite{ACFN1,ACFN4,ACFN5}, and subsequently \cite{ACFN6}) concern the study of the so-called \emph{infinite $N$-star} graph (see Figure \ref{fig-nstar}), with boundary conditions of $\delta$\emph{-type} or $\delta'$\emph{-type}.

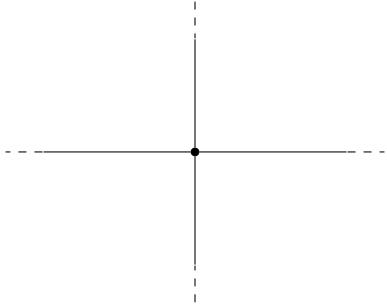
\begin{figure}
 \centering
\begin{tikzpicture}
\node at (0,0) [nodo] (00) {};
\node at (2,0) [infinito]  (20) {};
\node at (2.5,0) [infinito]  (250) {};
\node at (-2,0) [infinito]  (-20) {};
\node at (-2.5,0) [infinito]  (-250) {};
\node at (0,1.5) [infinito]  (015) {};
\node at (0,2) [infinito]  (02) {};
\node at (0,-1.5) [infinito] (0-15) {};
\node at (0,-2) [infinito] (0-2) {};
\draw [-] (00) -- (20);
\draw [dashed] (20) -- (250);
\draw [-] (00) -- (-20);
\draw [dashed] (-20) -- (-250);
\draw [-] (00) -- (015);
\draw [dashed] (02) -- (015);
\draw [-] (00) -- (0-15);
\draw [dashed] (0-2) -- (0-15);
\end{tikzpicture}
\caption{infinite $N$-star graph ($N=4$).}
\label{fig-nstar}
\end{figure}

On the other hand, in the case of \emph{Kirchhoff conditions}, that is functions with the sum of the derivatives equal to zero at the vertices (see \eqref{eq-kirch} below), more complex topologies have been managed (e.g., Figure \ref{fig-gen}). In \cite{AST2,AST3,AST4,LLS} there is a discussion of the existence of \emph{ground states}, namely solutions of \eqref{eq-statNLSgeneral} arising as global minimizers of the NLS energy functional
\begin{equation}
 \label{eq-encomp}
 E(u):=\frac12\int_\G|u'|^2\dx-\frac1p\int_\G|u|^p\dx
\end{equation}
among functions with fixed mass $\mu>0$, i.e. $\int_\G|u|^2\dx=\mu$. Precisely, \cite{AST2,AST3} investigate the so-called $L^2$-subcritical regime $p\in(2,6)$, while \cite{AST4} treats the critical case $p=6$. Furthermore, in \cite{AST6,CFN1,NPS,NRS} the investigation has been extended to more general stationary solutions that do not necessarily minimize the energy functional.

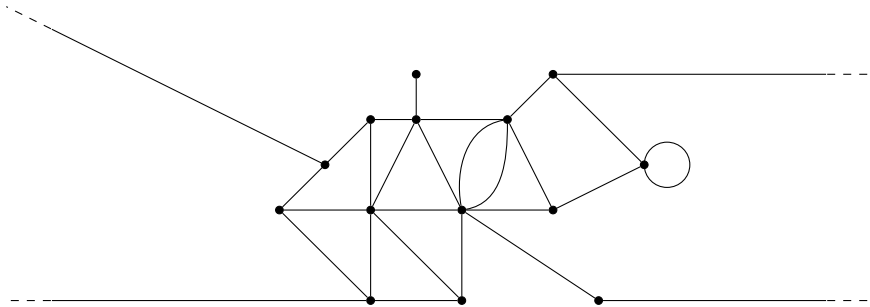
\begin{figure}
\centering
\begin{tikzpicture}[xscale= 0.6,yscale=0.6]
\node at (0,0) [nodo] (1) {};
\node at (2,0) [nodo] (2) {};
\node at (4,1) [nodo] (3) {};
\node at (2,3) [nodo] (4) {};
\node at (1,2) [nodo] (5) {};
\node at (-1,2) [nodo] (6) {};
\node at (-2,0) [nodo] (7) {};
\node at (-1,3) [nodo] (8) {};
\node at (-2,2) [nodo] (9) {};
\node at (-4,0) [nodo] (10) {};
\node at (-2,-2) [nodo] (11) {};
\node at (0,-2) [nodo] (12) {};
\node at (3,-2) [nodo] (13) {};
\node at (-3,1) [nodo] (14) {};
\node at (8,-2) [infinito] (15) {};
\node at (9,-2) [infinito] (15b) {};
\node at (8,3) [infinito] (16) {};
\node at (9,3) [infinito] (16b) {};
\node at (-9,-2) [infinito] (17) {};
\node at (-10,-2) [infinito] (17b) {};
\node at (-9,4) [infinito] (18) {};
\node at (-10,4.5) [infinito] (18b) {};
\draw (4.5,1) circle (0.5cm);
\draw [-] (1) -- (2);
\draw [-] (2) -- (3);
\draw [-] (3) -- (4);
\draw [-] (4) -- (5);
\draw [-] (5) -- (2);
\draw [-] (1) -- (13);
\draw [-] (1) -- (12);
\draw [-] (5) -- (6);
\draw [-] (1) -- (6);
\draw [-] (8) -- (6);
\draw [-] (11) -- (12);
\draw [-] (7) -- (12);
\draw [-] (1) -- (7);
\draw [-] (6) -- (7);
\draw [-] (9) -- (7);
\draw [-] (7) -- (11);
\draw [-] (11) -- (10);
\draw [-] (7) -- (10);
\draw [-] (9) -- (10);
\draw [-] (9) -- (6);
\draw [-] (1) to [out=100,in=190] (5);
\draw [-] (1) to [out=10,in=270] (5);
\draw [-] (13) -- (15);
\draw [dashed] (15) -- (15b);
\draw [-] (4) -- (16);
\draw [dashed] (16) -- (16b);
\draw [-] (14) -- (18);
\draw [dashed] (18) -- (18b);
\draw [-] (11) -- (17);
\draw [dashed] (17) -- (17b);
\end{tikzpicture}
\caption{a general noncompact metric graph.}
\label{fig-gen}
\end{figure}

\medskip
A modification of this model, proposed e.g. by \cite{GSD,N}, arises when one assumes that the nonlinearity affects only the \emph{compact core} $\K$ of the graph, namely the subgraph consisting of all its bounded edges (e.g., the compact core of Figure \ref{fig-nstar} is empty, while the one of Figure \ref{fig-gen} is given by Figure \ref{fig-gencomp}). In this case, the stationary equation of interest reads as
\begin{equation}
\label{eq-statNLSconcentrated}
	 u''+\chi_{\K}|u|^{p-2}\,u=\lambda u\qquad\text{(+ Kirch. cond.)}\,,
\end{equation}

\noindent with $\chi_{\K}$ denoting the characteristic function of $\K$.

The existence of solutions to this problem has been widely investigated in the $L^2$-subcritical case in \cite{ST1,ST2,T}. In particular, \cite{T} discusses the existence of the ground states of the modified energy functional
\begin{equation}
 \label{eq-en}
 E(u,\K):=\frac12\int_\G|u'|^2\dx-\frac1p\int_\K|u|^p\dx,
\end{equation}
while \cite{ST1,ST2} manage more general stationary solutions.

In this paper we aim at giving a first insight on the existence/nonexistence of ground states of the problem with the \emph{localized nonlinearity} in the critical case $p=6$. In particular, as a preliminary study, we explore a specific graph, the \emph{tadpole} graph (see Figure \ref{fig-tadpole}), which allows to point out some peculiar features of the problem whose understanding will suggest interesting perspectives for future investigations. More precisely, in our main result (namely, Theorem \ref{thm-main}) we prove, first, that there exists a threshold mass $\mu_1$ under which ground states cannot exist even though the functional $E(\cdot,\K)$ is bounded from below. Therefore, we establish the existence of another threshold $\mu_2\geq\mu_1$ such that, if $\mu\in[\mu_2,\mu_\rr]$ (where $\mu_\rr$ is the \emph{critical mass} of the real line defined by \eqref{eq-critreal}), then a ground state does exist; and, finally, that, for all $\mu>\mu_\rr$, $E(\cdot,\K)$ is unbounded from below.

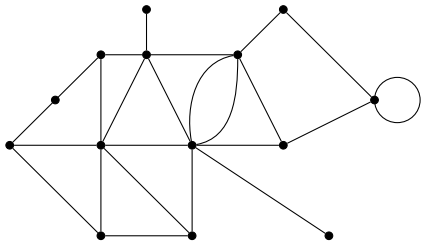
\begin{figure}
\centering
\begin{tikzpicture}[xscale= 0.6,yscale=0.6]
\node at (0,0) [nodo] (1) {};
\node at (2,0) [nodo] (2) {};
\node at (4,1) [nodo] (3) {};
\node at (2,3) [nodo] (4) {};
\node at (1,2) [nodo] (5) {};
\node at (-1,2) [nodo] (6) {};
\node at (-2,0) [nodo] (7) {};
\node at (-1,3) [nodo] (8) {};
\node at (-2,2) [nodo] (9) {};
\node at (-4,0) [nodo] (10) {};
\node at (-2,-2) [nodo] (11) {};
\node at (0,-2) [nodo] (12) {};
\node at (3,-2) [nodo] (13) {};
\node at (-3,1) [nodo] (14) {};
\draw (4.5,1) circle (0.5cm);
\draw [-] (1) -- (2);
\draw [-] (2) -- (3);
\draw [-] (3) -- (4);
\draw [-] (4) -- (5);
\draw [-] (5) -- (2);
\draw [-] (1) -- (13);
\draw [-] (1) -- (12);
\draw [-] (5) -- (6);
\draw [-] (1) -- (6);
\draw [-] (8) -- (6);
\draw [-] (11) -- (12);
\draw [-] (7) -- (12);
\draw [-] (1) -- (7);
\draw [-] (6) -- (7);
\draw [-] (9) -- (7);
\draw [-] (7) -- (11);
\draw [-] (11) -- (10);
\draw [-] (7) -- (10);
\draw [-] (9) -- (10);
\draw [-] (9) -- (6);
\draw [-] (1) to [out=100,in=190] (5);
\draw [-] (1) to [out=10,in=270] (5);
\end{tikzpicture}
\caption{the compact core of the graph in Figure \ref{fig-gen}.}
\label{fig-gencomp}
\end{figure}

\medskip
 For the sake of completeness we also mention some other recent works on the stationary solutions of the NLS equation on graphs. Problems with a wide class of $\delta$-type conditions and external potentials are managed in \cite{C,CFN2}. On the other hand, \cite{Do,MP} discuss compact graphs, while \cite{AD,ADST,GPS,PS} focus on periodic graphs (i.e., graphs whose noncompactness is not due to the presence of half-lines, but to the infinite number of edges). Finally, it is worth quoting three further works on evolution equations on graphs. The former is \cite{Du}, where is presented a preliminary result of Control Theory on graphs for the bi-linear Schr\"odinger equation; then \cite{MNS} introduces the study of the Airy equation (thus opening to the application of metric graphs in hydrodynamics); and, finally, \cite{BCT} discusses the bound states of another important dispersive equation on graph, the NonLinear Dirac (NLD) equation.

\medskip
The paper is organized as follows. In Section \ref{sec-res} we present a precise setting of the problem and we state our main result (Theorem \ref{thm-main}). In Section \ref{sec-comp} we show some preliminary results, mainly concerning compactness issues, while Section \ref{sec-proof} provides the proof of the main theorem of the paper.


\section{Setting and main results}
\label{sec-res}

We consider the \emph{tadpole} graph $\T$ (Figure \ref{fig-tadpole}), that is a connected noncompact metric graph consisting of a compact circle $\K$ and a half-line $\h$ (endowed with the usual intrinsic parametrization -- see \cite{AST2}) incident at the vertex $\mathrm{v}$.

\begin{figure}
 \centering
 \includegraphics[width=0.5\textwidth]{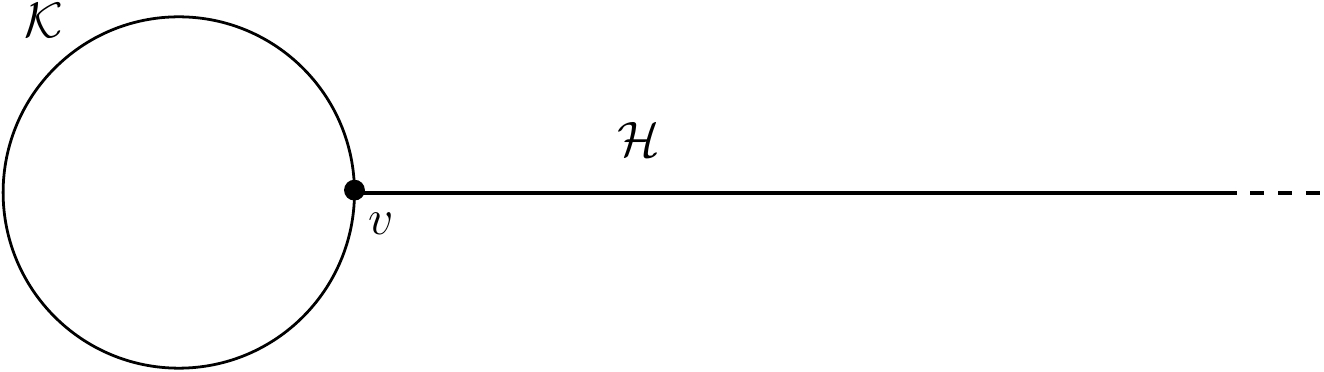}
 \caption{a tadpole graph.}
 \label{fig-tadpole}
\end{figure}

A function $u:\T\to\rr$ can be seen as a couple of functions $(v,w)$, with $v:\K\to\rr$ and $w:\h\to\rr$, and thus Lebesgue and Sobolev spaces can be defined as usual
\[
 L^p(\T):=\{u:\T\to\rr:v\in L^p(\K),\,w\in L^p(\h)\}
\]
and
\[
 H^1(\T):=\{u:\T\to\rr\, \text{ continuous}:v\in H^1(\K),\,w\in H^1(\h)\}.
\]
We also define, for $\mu>0$,
\[
 H_\mu^1(\T):=\left\{u\in H^1(\T):\int_\T|u|^2\dx=\mu\right\}\,.
\]

We address the problem of the existence of a function $u\in H_\mu^1(\T)$ such that $E(u,\K)=\mathcal{E}_\K(\mu)$, where
\begin{equation}
 \label{eq-energy_defn}
 E(u,\K):=\frac{1}{2}\int_{\T}|u'|^2\,dx-\frac{1}{6}\int_{\K}|u|^6\,dx
\end{equation}
and
\begin{equation}
 \label{eq-ground_lev}
 \mathcal{E}_\K(\mu):=\inf_{u\in\HmuT}E(u,\K).
\end{equation}
It is clear that such a minimizer $u$, usually called \emph{ground state}, satisfies
\[
 \left\{
 \begin{array}{l}
  \displaystyle v''+|v|^4\,v=\lambda v\\[.2cm]
  \displaystyle w''=\lambda w
 \end{array}
 \right.
\]
(for some $\lambda>0$) and
\begin{equation}
 \label{eq-kirch}
 v'(0)-v'(L)+w'(0)=0
\end{equation}
where $L:=|\K|$ and we have considered an anti-clockwise parametrization of $\K$, i.e. $u$ solves the stationary NLS equation \eqref{eq-statNLSconcentrated} on $\T$.

\begin{rem}
 We limit ourselves to consider real valued functions in the search of ground states since it can be shown that minimizers of the NLS energy are always real valued up to the multiplication times a constant phase (for more see \cite{AST2,T}). It is also possible to prove, by an easy regularity argument, that a ground state cannot be equal to zero at any point of the graph.
\end{rem}

\medskip
Before stating the main result of the paper, it is worth recalling some well-known facts on the ground states of the complete problem, i.e. with the nonlinearity extended on the whole graph. Precisely, we have to introduce the concept of \emph{critical mass}, as the existence of a minimizer in the critical case $p=6$ is strictly connected to the value of the mass. 

When $\G=\rr$ (see \cite{Ca}),
\begin{equation}
\label{eq-critreal}
 \inf_{u\in H_\mu^1(\rr)}E(u)=\left\{
 \begin{array}{ll}
  \displaystyle 0 & \text{if }\mu\leq\mu_\rr\\[.3cm]
  \displaystyle -\infty & \text{if }\mu>\mu_\rr
 \end{array}
 \right.
 \qquad(\mu_\rr=\frac{\sqrt{3}}{2}\pi)
\end{equation}
and the infimum is attained only at $\mu=\mu_\rr$; whereas, when $\G=\rr^+$,
\[
 \inf_{u\in H_\mu^1(\rr^+)}E(u)=\left\{
 \begin{array}{ll}
  \displaystyle 0 & \text{if }\mu\leq\mu_{\rr^+}\\[.3cm]
  \displaystyle -\infty & \text{if }\mu>\mu_{\rr^+}
 \end{array}
 \right.
 \qquad(\mu_{\rr^+}=\frac{\sqrt{3}}{4}\pi)
\]
and the infimum is attained only at $\mu=\mu_{\rr^+}$. The values $\mu_\rr$ and $\mu_{\rr^+}$ are said to be the critical masses of the line and of the half-line (respectively).

Concerning the tadpole $\G=\T$, as a consequence of Theorem 3.3 in \cite{AST4}, it has been proved that

\begin{equation*}
	\inf_{u\in\HmuT}E(u)\begin{cases}
	\geq0 & \text{if }\mu\leq\mu_{\rr^+}\\
	<0 & \text{if }\mu\in(\mu_{\rr^+},\mu_\rr]\\
	=-\infty & \text{if }\mu>\mu_\rr
	\end{cases}
\end{equation*}

\noindent and global minimizers of the energy exist if and only if $\mu\in(\mu_{\rr^+},\mu_\rr]$.

\medskip
We can now present the main result of this paper.

\begin{thm}
 \label{thm-main}
 There exist two values $\mu_1,\mu_2\in(\mu_{\rr^+},\mu_{\rr})$, with $\mu_1<\mu_2$, such that
 \begin{itemize}
  \item[(i)] if $\mu\leq\mu_1$, then $\mathcal{E}_\K(\mu)=0$ and it is not attained;
  \item[(ii)] if $\mu_2\leq\mu\leq\mu_\rr$, then there exists a ground state of mass $\mu$.
  \item[(iii)] if $\mu>\mu_\rr$, then $\mathcal{E}_\K(\mu)=-\infty$.
 \end{itemize}
 Furthermore, ground states always realize strictly negative energy levels.
\end{thm}
    
We point out that the previous result displays a different phenomenology with respect to the analogous in the everywhere nonlinear problem. Indeed, even though ground states are proved to exist only for some intervals of masses in both cases, Theorem \ref{thm-main} suggests that these intervals are actually different. Specifically, it appears that concentrating the nonlinearity on the compact core does not allow the presence of global minimizers if the mass is too close to $\mu_{\rr^+}$, and a new lower threshold must arise. However, we are not able to detect the sharp values of $\mu_1$ and $\mu_2$ at the moment, so that it is  still an open problem to determine $\mu^*$ such that ground states with concentrated nonlinearity exist if and only if $[\mu^*,\mu_\rr]$. We will address this issue in a forthcoming paper.   


\section{Preliminaries and compactness}
\label{sec-comp}

First, let us recall some previous results on noncompact metric graphs, highlighting the consequences they have on the problem we discuss in the paper.

It is well-known (see for instance \cite{AST2,T}) that the following Gagliardo-Nirenberg inequalities
\begin{equation}
 \label{eq-GN}
 \uLsixT^6\leq C_\T\uLtwoT^4\udotT^2,
\end{equation}
\begin{equation}
 \label{eq-GN_infty}
 \|u\|_{L^\infty(\T)}\leq C_\infty\uLtwoT^{1/2}\udotT^{1/2}
\end{equation}
hold for every $u\in H^1(\T)$ (here $C_\T, C_\infty$ denote the optimal constants). Furthermore, a modified version of \eqref{eq-GN} has been established in \cite[Lemma 4.4]{AST4}. Precisely, for every $u\in\HmuT$, there exists $\theta_u:=\theta(u)\in[0,\mu]$, such that
\begin{equation}
 \label{eq-modified_GN}
 \uLsixT\leq 3\Big(\frac{\mu-\theta_u}{\mu_\rr}\Big)^2\udotT^2+C\sqrt{\theta_u}
\end{equation}
with $C>0$ independent of $u$.

In addition, recalling the definition of the complete NLS energy given by \eqref{eq-encomp} with $\G=\T$, we know (again from \cite{AST4}) that 
\[
\begin{array}{llll}
 \displaystyle (i)   & \displaystyle \mu\leq\mu_{\rr^+}         & \displaystyle \quad\Longrightarrow\quad & \displaystyle E(u)>0,\qquad\forall u\in\HmuT;\\[.7cm]
 \displaystyle (ii)  & \displaystyle \mu_{\rr^+}<\mu\leq\mu_\rr & \displaystyle \quad\Longrightarrow\quad & \displaystyle -\infty<\inf_{u\in\HmuT}E(u)<0\\[.7cm]
 \displaystyle (iii) & \displaystyle \mu<\mu_\rr                & \displaystyle \quad\Longrightarrow\quad & \displaystyle \inf_{u\in\HmuT}E(u)=-\infty. 
\end{array}
\]
Moreover, global minimizers exist only in case $(ii)$.

Since it is straightforward that, for every $u\in H^1(\T)$,
\[
 E(u,\K)\geq E(u),
\]
the previous observations have some relevant consequences on the problem with localized nonlinearity too. In fact, we have that
\begin{equation}
 \label{eq-enpos}
 E(u,\K)>0,\qquad\forall u\in u\in\HmuT,
\end{equation}
for every $\mu\leq\mu_\rr^+$, and that
\begin{equation}
 \label{eq-lowbound}
 \mathcal{E}_\K(\mu)>-\infty,\qquad\forall\mu\in(\mu_{\rr^+},\mu_\rr].
\end{equation}
On the other hand, arguing exactly as in  \cite{AST4}, one can show that,
\[
 \mathcal{E}_\K(\mu)=-\infty,\qquad\forall\mu>\mu_\rr,
\]
which immediately proves item \emph{(iii)} of Theorem \ref{thm-main}.


\medskip
We conclude this section establishing a compactness result, valid only for localized nonlinearities, which ensures that ground states exist if and only if the infimum of the energy is strictly negative and finite.

\begin{lem}
\label{lem-comp}
 Let $\mu\in(0,\mu_{\rr}]$. If
 \begin{equation}
  \label{eq-inf_neg}
  -\infty<\mathcal{E}_\K(\mu)<0
 \end{equation}
 then, there exists a ground state of $E(\,\cdot\,,\K)$ of mass $\mu$. If, on the contrary, $E(u,\K)>0$ for every $u\in\HmuT$, then 
 \begin{equation}
  \label{eq-infzero}
  \mathcal{E}_\K(\mu)=0
 \end{equation}
 and it is not attained.
\end{lem}

\begin{proof}
 We start by showing that, for every $\mu\in(0,\mu_\rr]$,
 \begin{equation}
  \label{eq-infnp}
  \mathcal{E}_\K(\mu)\leq0\,.
 \end{equation}
 For every $n\in\mathbb{N}$, define 
 \[
  u_n(x):=\begin{cases}
  \alpha_n,         & \text{if }x\in (1,n)\cap\h,\\[.2cm]
  \alpha_n x,       & \text{if }x\in[0,1]\cap\h,\\[.2cm]
  \alpha_n (n+1-x), & \text{if }x\in[n,n+1]\cap\h,\\[.2cm]
  0,                & \text{elsewhere on }\T,
  \end{cases}
 \] 
 where $\{\alpha_n\}_{n\in\mathbb{N}}$ is chosen so that $\|u_n\|_{L^2(\T)}^2=\mu$, for every $n\in\mathbb{N}$ (note that this entails $\alpha_n\to0$, as $n\to+\infty$). It is, then, easy to check that $u_n\to0$ strongly in $H^1(\T)$, thus implying that $E(u_n,\K)\to0$, as $n\to+\infty$, and hence that \eqref{eq-infnp} is satisfied.
 
 On the other hand, if $E(u,\K)>0$ for every $u\in\HmuT$, then \eqref{eq-infnp} yields \eqref{eq-infzero} and, consequently, the infimum cannot be attained.
 
 Finally, suppose that, on the contrary, \eqref{eq-inf_neg} holds and let $\{u_n\}_{n\in\mathbb{N}}\subset\HmuT$ be a minimizing sequence for $E(\,\cdot\,,\K)$. Then, for large $n$, $E(u_n,\K)\leq-c$, with $c>0$, and, combining with \eqref{eq-modified_GN}, this entails
 \[
  \frac{1}{2}\|u_n'\|_{L^2(\T)}^2\Big[1-\frac{(\mu-\theta_{u_n})^2}{\mu_{\rr}^2}\Big]-C\sqrt{\theta_{u_n}}\leq E(u_n,\G)\leq-c<0
 \]
 with $\theta_{u_n}\in[0,\mu]$. Thus, one finds that $\theta_{u_n}\geq\widetilde{c}>0$, so that $\left(\frac{\mu-\theta_{u_n}}{\mu_\rr}\right)^2<1$ and hence $\{u_n\}_{n\in\mathbb{N}}$ is bounded in $\HmuT$. As a consequence, $u_n\rightharpoonup u$ in $H^1(\T) $ and $u_n\to u$ in $L_{loc}^6(\T)$ (up to subsequences), and thus
 \[
  E(u,\K)\leq\liminf_n E(u_n,\K)=\mathcal{E}_\K(\mu).
 \]
 It is, then, left to prove that $\|u\|_{L^2(\T)}^2=:m=\mu$.
 
 First we see that, if $m=0$, then $u\equiv 0$, and hence 
 \[
  \mathcal{E}_K(\mu)=\liminf_{n}E(u_n,\K)\geq E(u,\K)=0,
 \]
 which contradicts \eqref{eq-inf_neg}. On the other hand, if $m<\mu$, then there exists $\sigma>1$ satisfying $\|\sigma u\|_{L^2(\T)}^2=\mu$. However, this implies that
 \[
  E(\sigma u,\K)=\frac{\sigma^2}{2}\int_{\T}|u'|^2\,dx-\frac{\sigma^6}{6}\int_{\K}|u|^6\,dx<\sigma^2E(u,\K)<E(u,\K),
 \]
 which is again a contradiction. Hence, $m=\mu$, which concludes the proof.
\end{proof}

    
\section{Proof of Theorem \ref{thm-main}}
\label{sec-proof}

This section is devoted to the proof of items \emph{(i)} and \emph{(ii)} of Theorem \ref{thm-main} (item \emph{(iii)} has been already discussed in the previous section).
    
\begin{proof}[Proof of Theorem \ref{thm-main}: item (i)]
 First, note that, whenever $\mu\leq\mu_{\rr^+}$, combining \eqref{eq-enpos} and Lemma \ref{lem-comp}, one easily sees that no ground state may exist. 
 
 On the other hand, assume (by contradiction) that there exists a ground state of $E(\,\cdot\,,\K)$ of mass $\mu$, for every $\mu\in(\mu_{\rr^+},\mu_\rr]$. Then, let $\{\mu_\varepsilon\}_{\varepsilon>0}$ be a sequence such that $\mu_\varepsilon\to\mu_{\rr^+}$, as $\varepsilon\to0$, and let $u_\varepsilon$ be (one of) the associated ground state(s).
    
 Now, it is immediate that $E(u_\varepsilon,\K)\leq0$. As a consequence, exploiting the modified Gagliardo-Nirenberg inequality \eqref{eq-modified_GN} as in the proof of Lemma \ref{lem-comp}, there results 
 \[
   \|u_\varepsilon'\|_{L^2(\T)}^2\leq\frac{\mu_\varepsilon^2}{\mu_\rr^2}\|u_\varepsilon'\|_{L^2(\T)}^2+C\sqrt{\mu_\rr}.
 \]
 Therefore $\{u_\varepsilon\}_{\varepsilon>0}$ is bounded in $H^1(\T)$ and there exists $u\in H^1(\T)$ such that $u_\varepsilon\rightharpoonup u$ in $H^1(\T)$ and $u_\varepsilon\to u$ in $L_{loc}^6(\T)$ (up to subsequences), as $\varepsilon\to0$. 
    
 Furthermore, using \eqref{eq-GN_infty} and (again) the negativity of the energy, one finds
 \begin{align*}
  \|u_\varepsilon'\|_{L^2(\T)}^2 < & \, \frac{1}{3}\|u_\varepsilon\|_{L^6(\K)}^6\leq\frac{L}{3}\|u_\varepsilon\|_{L^\infty(\K)}^6\leq\frac{L}{3}\|u_\varepsilon\|_{L^\infty(\T)}^6\\[.3cm]
                              \leq & \, \frac{C_\infty^6 L}{3}\|u_\varepsilon\|_{L^2(\T)}^3\|u_\varepsilon'\|_{L^2(\T)}^3=\frac{C_\infty^6 L}{3}\mu_\varepsilon^{3/2}\|u_\varepsilon'\|_{L^2(\T)}^3,
 \end{align*}
 which yields (as $\|u_\varepsilon'\|_{L^2(\T)}^2\neq0$)
 \begin{equation}
  \label{eq-inf_kine}
  \|u_\varepsilon'\|_{L^2(\T)}\geq\frac{3}{C_\infty^6 L\mu_\rr^{3/2}},\qquad\forall\varepsilon>0,
 \end{equation}
 thus preventing $u\equiv 0$. Indeed, if $u\equiv 0$, then $u_\varepsilon\to0$ in $L^\infty(\K)$ (from compact embeddings) and, as $E(u_\varepsilon,\K)\leq0$, 
 \[
  \|u_\varepsilon'\|_{L^2(\T)}^2<\frac{1}{3}\|u_\varepsilon\|_{L^6(\K)}^6\leq\frac{1}{3}L\|u_\varepsilon\|_{L^\infty(\K)}^6\to0,\qquad\text{as }\varepsilon\to0,
 \] 
 but this contradicts \eqref{eq-inf_kine}.
 
 Finally, by the weak lower semicontinuity, we have
 \[
  \uLtwoT^2\leq\liminf_{\varepsilon\to0}\mu_\varepsilon=\mu_{\rr^+}
 \]
 and
 \[
  E(u,\K)\leq\liminf_{\varepsilon\to0}E(u_\varepsilon,\K)\leq0.
 \]
 Hence, $u$ is a function in $H_m^1(\T)$, for some $m\in(0,\mu_{\rr^+}]$, such that $E(u,\K)\leq0$. However, this is forbidden by \eqref{eq-enpos}, which (combining with Lemma \ref{lem-comp}) concludes the proof.
\end{proof}

\begin{proof}[Proof of Theorem \ref{thm-main}: item (ii)]
 Since by \eqref{eq-lowbound} the energy functional is lower bounded (whenever $\mu\leq\mu_\rr$), from Lemma \ref{lem-comp}, it is sufficient to exhibit a function with a strictly negative energy (as the mass exceeds a certain threshold).
 
 To this aim, fix $\mu\in(\mu_{\rr^+},\mu_\rr]$ and let    
 \begin{equation}
  \label{eq-const_exp}
  u(x):=\begin{cases}
  c & \text{if }x\in\K\\
  ce^{-\alpha x}  & \text{if }x\in\h,
  \end{cases}
 \end{equation}
 with $c,\alpha>0$ satisfying the mass condition
 \begin{equation}
  \label{eq-mass_alpha}
  \mu=\uLtwoT^2=\int_{\K}c^2\,dx+\int_{\h}c^2e^{-2\alpha x}\,dx=c^2L+\frac{c^2}{2\alpha}\,
 \end{equation}
 (see also Figure \ref{fig-21}).
 
 \begin{figure}
  \centering
  \includegraphics[width=0.7\textwidth]{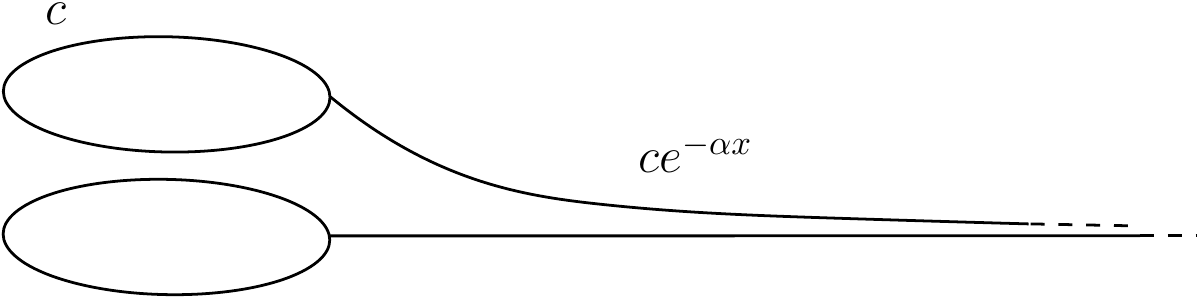}
  \caption{function introduced in the proof of Theorem \ref{thm-main}: item (ii).}
  \label{fig-21}
 \end{figure}
    
 Hence, $u\in\HmuT$ and its energy reads
 \begin{equation}
  \label{eq-test_energy}
  E(u,\K)=\frac{c^2\alpha}{4}-\frac{c^6L}{6}.
 \end{equation}
 Now, by \eqref{eq-mass_alpha}
 \[
  \alpha=\frac{c^2}{2(\mu-c^2L)},
 \]
 so that
 \[
  E(u,\K)=\frac{c^4}{8(\mu-c^2L)}-\frac{c^6L}{6}=\frac{c^4}{2}\Big(\frac{1}{4(\mu-c^2L)}-\frac{c^2L}{3}\Big).
 \]
 Therefore, imposing $E(u,\K)<0$ reduces to determine whether exists (or not) a value $c$ such that
 \[
  \frac{1}{4(\mu-c^2L)}-\frac{c^2L}{3}<0,
 \]
 namely, whether exists (or not) a value $\Lambda:=c^2L$ such that
 \[
  \Lambda^2-\mu\Lambda+\frac{3}{4}<0.
 \]
 However, the previous inequality is satisfied whenever
 \[
  \frac{\mu-\sqrt{\mu^2-3}}{2}<\Lambda<\frac{\mu+\sqrt{\mu^2-3}}{2},
 \]
 provided that
 \begin{equation}
  \label{eq-squareroot}
  \mu^2-3\geq0\qquad\Leftrightarrow\qquad\mu\geq\sqrt{3}\,.
 \end{equation}
 Henceforth, setting $\mu_2:=\sqrt{3}$, for every $\mu\in[\mu_2,\mu_\rr]$, there exists $u\in\HmuT$ such that $E(u,\K)<0$.
\end{proof}



\end{document}